\documentclass[10pt, reqno]{amsart}

\usepackage{amsfonts}
\usepackage{amsthm}
\usepackage{mathrsfs}
\usepackage{amsmath}
\usepackage{hyperref}

\theoremstyle{theorem}
\newtheorem{theorem}{Theorem}
\newtheorem{corollary}[theorem]{Corollary}
\newtheorem{lemma}[theorem]{Lemma}
 
\theoremstyle{definition}

\newtheorem*{remark}{Remark}

\newcommand{\M}{\mathbb{M}}
\newcommand{\C}{\mathbb{C}}
\newcommand{\R}{\mathbb{R}}

\DeclareMathOperator{\re}{Re}
\DeclareMathOperator{\tr}{tr}
\DeclareMathOperator{\dist}{dist}

\DeclareMathOperator{\ran}{ran}

\begin{document}
\title{Polynomially Isometric Matrices in Low Dimensions}
\author{Cara D. Brooks, Alberto A. Condori, and Nicholas Seguin}
\maketitle

\begin{abstract}
			Given two $d\times d$ matrices, say $A$ and $B$, when do $p(A)$ and $p(B)$ have the same 
			``size'' for every polynomial $p$?  In this article, we provide definitive results in the
			cases $d=2$ and $d=3$ when the notion of size used is the spectral norm.
\end{abstract}

\section{Introduction.}\label{introSection} 

Given a square matrix $A$, there is no ambiguity in what ``squaring a matrix'' 
should mean; $A^2$ is the product of $A$ with itself.  This simple notion can 
be extended in a natural way to nonnegative integer powers of matrices: $A^n$ 
is the product of $n$ copies of $A$ when $n>0$, while $A^0$ is the identity 
matrix $I$. Likewise, one may construct other matrices associated with $A$, 
namely, polynomial functions $p(A)$ of $A$.  That is, given a polynomial 
with complex coefficients $p(z)=c_0+c_1 z+\cdots+c_m z^m$, $p(A)$ denotes the 
square matrix $c_0 I+c_1 A+\cdots+c_m A^m$.  This definition also works just 
as well if $A$ is an operator (i.e., a linear transformation) on a complex 
vector space; the only difference is that the operation of composition is 
used instead of matrix product.  The assignment $p\mapsto p(A)$ induced by 
$A$ is often referred to as its polynomial functional calculus.

Where do polynomial functions of a matrix or operator come about, and why 
are they important?  Although answers abound, let us mention only a couple 
of places where they are encountered, perhaps in disguise, in the undergraduate 
curriculum.  For other types of functional calculus and their applications, we 
refer the interested reader to \cite[Chapter 10]{Ru} for the Riesz--Dunford 
holomorphic functional calculus (for the graduate student) and to \cite{N} 
for a discussion of the Dyn'kin nonholomorphic functional calculus (for the 
advanced scholar).

First, a fundamental result in linear algebra states that every operator $A$ 
on a finite-dimensional complex vector space (e.g., a square matrix) has an 
eigenvalue.  (As usual, $\lambda$ is an eigenvalue of $A$ if there is a 
nonzero vector $v$ so that $Av=\lambda v$.)  A quick proof of that result 
amounts to observing that given a nonzero vector $w$, there is a polynomial 
$p$ so that $p(A)w=0$ (see \cite[p. 145]{Ax} for details).  Second, a linear 
inhomogeneous differential equation (DE) with constant coefficients can 
be seen as an operator equation of the differentiation operator $D$.  For 
instance, the DE 
\begin{equation}\label{DEex}
			y^{\prime\prime}(t)+c_1 y^\prime(t)+c_0y(t)=f(t)
\end{equation}
can be written as $p(D)y=f$, where $p(z)=z^2+c_1z+c_0$ is a quadratic polynomial 
and $y$ is the unknown function.  In this case, if $p_j(z)=z-\lambda_j$ for $j=1,2$ 
are such that $p_1(z)p_2(z)=p(z)$, then the functional calculus for $D$ gives a way 
to solve \eqref{DEex}: solve consecutively the first order DEs 
$y_1^\prime-\lambda_1 y_1=p_1(D)y_1=f$ for $y_1$ and 
$y^\prime-\lambda_2 y=p_2(D)y=y_1$ for $y$.

In applications, not only is a function of a matrix important, so is its “size.”
For instance, when $A$ is a square matrix, $\|A^k\|$ and $\|(zI-A)^{-1}\|$ arise 
naturally in the models of discrete-time evolution processes and responses 
of forced systems, respectively \cite[Chapter 47]{TE}. Such quantities are used to 
describe behavior, and so it is natural to question what conditions a matrix $B$ 
might satisfy to ensure that its behavior is the same as that of $A$.   

More precisely and following \cite{MZ}, we say that $A$ and $B$ are 
\textbf{polynomially isometric}\footnote{The terminology used in this 
article was introduced in \cite{MZ}.  However, the same notion has 
appeared previously as ``$A$ and $B$ have the same norm behavior,'' 
e.g., see \cite{GT} and \cite[Chapter 47]{TE}.} (under the spectral 
norm) if
\begin{equation}\label{condP}
			\|p(A)\|=\|p(B)\|\quad\text{ for \emph{all} polynomials } p. 
\end{equation}

Thus, in this article, we consider the following question:
\begin{quote}
			{\normalsize Given a pair of square matrices $A$ and $B$, what 
			set of invariants (e.g., spectra, Frobenius norms, etc.) are 
			necessary and sufficient to ensure that $A$ and $B$ are 
			polynomially isometric?}
\end{quote}

One might think that unitary similarity characterizes \eqref{condP}, 
but the condition turns out to be too strong. Unitary similarity is 
certainly \emph{sufficient} to ensure that two matrices are polynomially 
isometric; after all, if there is a unitary matrix $U$ (i.e., $U^{*}U=UU^{*}=I$) 
so that $B=UAU^{*}$, then $\|p(B)\|=\|p(A)\|$ holds for every polynomial $p$ 
since  $p(B)=Up(A)U^{*}$.  Furthermore, unitary similarity is also necessary 
for \eqref{condP} to hold if the matrices $A$ and $B$ are $2\times 2$ (see 
Theorem \ref{equivalenceM2} below).  However, if 
\begin{equation}\label{equalSpectraWithoutMultEx}
			A=\left[
			\begin{array}{ccc}
						1	&	0	& 0\\
						0	&	0	& 0\\
						0	&	0	&	0
			\end{array}\right]\;\text{ and }\;
			B=\left[
			\begin{array}{ccc}
						1	&	0	& 0\\
						0	&	1	& 0\\
						0	&	0	&	0
			\end{array}\right]
\end{equation}
then $\|p(A)\|=\max\{|p(1)|,|p(0)|\}=\|p(B)\|$ holds for all polynomials $p$, i.e., 
$A$ and $B$ are polynomially isometric, but $A$ and $B$ cannot be unitarily similar 
because they have different ranks. 
 
One might then turn to equality of spectra.  After all, equality of spectra is 
a \emph{necessary} condition for \eqref{condP} and furthermore, by the spectral 
theorem,  $\|p(N)\|=\max\{|p(\lambda)|:\lambda \in \sigma(N)\}$ holds whenever 
$N$ is a normal matrix (i.e., $N^{*}N = NN^{*}$).  However, the condition is not 
sufficient (see \eqref{sameMinPolEx} below) and this is not an isolated case; 
numerical analysts have long known that knowledge of the spectrum of a matrix 
\emph{alone} is not enough to describe the behavior of nonnormal matrices.  On 
the other hand, pseudospectral analysis has proven to be a useful tool to better 
understand the behavior of matrices that arise in scientific applications; i.e., 
matrices that are nonnormal and of large dimension.  For instance, the reader can 
find a wealth of examples in the book \cite{TE} (see also \cite{T}) that illustrate 
how pseudospectra\footnote{Roughly, a pseudospectral plot for a matrix $A$ consists 
of contour plots of the norm $\|(zI-A)^{-1}\|$ of its resolvent $(zI-A)^{-1}$.} 
may capture the ``spirit'' of a (nonnormal) matrix more effectively.  Unfortunately, 
it is observed in \cite{GT} that for \eqref{condP} to hold, it is necessary, but not 
sufficient, that $A$ and $B$ have identical pseudospectra (see \eqref{condR} below).

But hope is not lost! The condition of identical pseudospectra does not suffice in 
general, but for square matrices of small dimension, we show below that it does the 
trick.  We address the question in the context of $2\times 2$ and $3\times 3$ matrices, 
and prove the necessity and sufficiency of identical pseudospectra for \eqref{condP} 
to hold.  Not only does this article serve as ``food for thought'' for the linear-algebra 
enthusiast, it is aimed to provide clarity for newcomers to matrix analysis concerning 
the precise connections between some related notions encountered in the field, namely, 
polynomially isometric matrices, identical pseudospectra, super-identical pseudospectra, 
and unitary similarity.  

\section{Terminology and the main result.}\label{terminologySection}

Let $\C^{d}$ denote complex Euclidean $d$-dimensional space, and 
let $\M_{d}$ be the algebra of complex $d\times d$ matrices.  For
$T\in\M_d$, $\tr T$ denotes the trace of $T$.  We denote the Frobenius 
(or Hilbert--Schmidt) norm of $T$ by $\|T\|_{F}$ and the spectral norm 
of $T$ by $\|T\|$.  That is, $\|T\|_{F}=\sqrt{\tr T^{*}T}$, where $T^{*}$ 
is the conjugate transpose of $T$, and 
$\|T\|=\sup\{\|Tv\|_{\C^{d}}: \|v\|_{\C^{d}}=1\}$ is the operator norm 
induced by the Euclidean norm on $\C^{d}$.  The minimal and characteristic 
polynomials of $T$ are denoted by $m_{T}$ and $\chi_T$, respectively.  That is,  
$m_T$ is the monic polynomial $p$ of minimal degree such that $p(T)=0$ while
\[	\chi_T(z)=\det(zI-T),	\]
where the determinant $\det(A)$ is the product of the eigenvalues of $A\in\M_d$
(taking into account multiplicities).  
As usual, the spectrum $\sigma(T)$ of $T$ is the set of eigenvalues of $T$, i.e.,
\[	\sigma(T)=\{\lambda\in\C: \lambda I-T\text{ is not invertible}\}.	\]
Finally, the singular values $s_1(T), \ldots, s_{d}(T)$ of $T$ are the 
nonnegative square roots of the eigenvalues of $T^{*}T$ listed in nonincreasing 
order.  Thus, $s_1(T)=\|T\|$, $\|T\|_{F}^2=s_{1}^{2}(T)+\cdots+s_{d}^{2}(T)\geq \|T\|^2$, 
and $s_{d}(T)=\|T^{-1}\|^{-1}$ whenever $T$ is invertible.  We refer the reader 
to \cite{Ax} and \cite{H} for further explanations and results concerning these 
concepts.

At this point, one may be wondering what can (and cannot) be expected of 
polynomially isometric matrices $A$ and $B$. Surely, they need not have 
the same characteristic polynomials.  This is demonstrated by the pair 
of matrices in \eqref{equalSpectraWithoutMultEx}. Must they have the same 
spectra? Absolutely.  In fact, \eqref{condP} implies (for matrices $A$ and 
$B$ of arbitrary size) that the minimal polynomials $m_{A}$ and $m_{B}$ must 
be equal.

On the other hand, equality of minimal polynomials is not enough to guarantee 
the converse: the matrices
\begin{equation}\label{sameMinPolEx}
			A=\left[
			\begin{array}{ccc}
						0	&	0	& 0\\
						1	&	0	& 0\\
						0	&	0	&	0
			\end{array}\right]\;\text{ and }\;
			B=\left[
			\begin{array}{ccc}
						0	&	0	& 0\\
						1	&	0	& 0\\
						2	&	0	&	0
			\end{array}\right]
\end{equation}
have minimal polynomial $m_{A}(z)=m_{B}(z)=z^2$ but
\[	\|A\|\leq\|A\|_{F}=1\;\text{ while }\;
		\|B\|\geq\|(0,1,2)\|_{\C^{3}}=\sqrt{5}.	\]
Thus, \eqref{condP} fails with $p(z)=z$.

As mentioned in Section \ref{introSection}, it is known that having identical 
pseudospectra is also a necessary condition for matrices to be polynomially 
isometric \cite{GT}.  To be precise, let us agree that two square matrices 
$A$ and $B$ (\emph{not} necessarily of the same size) have \textbf{identical 
pseudospectra} if\footnote{In view of the well-known inequality 
$\|(zI-T)^{-1}\|\geq\dist^{-1}(z,\sigma(T))$, valid for $z\notin\sigma(T)$ (a 
consequence of Gelfand's spectral radius formula), we adopt the convention that 
$\|(zI-T)^{-1}\|=\infty$ for $z\in\sigma(T)$.  Thus, \eqref{condR} implies 
$\sigma(A)=\sigma(B)$.}
\begin{equation}\label{condR}
			\|(zI-A)^{-1}\|=\|(zI-B)^{-1}\|\quad\text{ for all }z\in\C. 
\end{equation}

Now since \eqref{condR} implies $m_{A}=m_{B}$ (see Theorem \ref{ipImpliesEqualMs} 
below), one may ask whether a pair of matrices are polynomially isometric precisely 
when they have identical pseudospectra.  If at least one of $A$ or $B$ is \emph{normal}, 
an affirmative answer is known \cite{BC}.  However, an example from \cite{GT} shows 
(after padding a matrix with zeros) that there are $5\times 5$ (nonnormal) matrices 
having identical pseudospectra for which the condition in \eqref{condP} fails with 
$p(z)=z$. (This example also appears in \cite[Chapter 47]{TE}.)  Furthermore, 
\cite{BR} and \cite{RR} contain examples of $4\times 4$ matrices having identical 
pseudospectra but whose squares have distinct norms, i.e., \eqref{condP} fails with 
$p(z)=z^2$.  Nevertheless, the following holds.
\begin{theorem}\label{MainThm}
			The following statements are equivalent for $A,B\in\M_d$ when $d=2,3$.
			\begin{enumerate}
					\item $A$ and $B$ have identical pseudospectra. 
					\item $A$ and $B$ are polynomially isometric. 
			\end{enumerate}
\end{theorem}

In the case $d=2$, the equivalence in Theorem \ref{MainThm} was mentioned (without 
proof) in \cite{GT} and so in Section \ref{sipSection}, we establish a slight 
improvement that either of the two statements listed in Theorem \ref{MainThm} 
is equivalent to saying that $A$ and $B$ are unitarily similar. In Section 
\ref{sipSection}, we also introduce the stronger notion of super-identical 
pseudospectra.  The main result of that section clarifies which $3\times 3$ 
matrices with identical pseudospectra also have super-identical pseudospectra.
This, coupled with the fact that matrices with identical pseudospectra have 
the same minimal polynomials, leads to the reduction of the proof of Theorem 
\ref{MainThm} to the case of matrices with quadratic minimal polynomials in 
Section \ref{quadMinPolySection}. Furthermore, in that context, we establish 
an easy-to-check necessary and sufficient condition in terms of the Frobenius 
norm that determines when a pair of matrices have identical pseudospectra.  
Finally, in Appendix \ref{appendix}, we include proofs of two technical results 
concerning $d\times d$ matrices having identical pseudospectra that are used in 
Sections \ref{sipSection} and \ref{quadMinPolySection}.

\section{Super-identical pseudospectra.}\label{sipSection}

Recall that polynomially isometric matrices must have identical pseudospectra 
(regardless of their size), but the converse need not hold.  Even more surprisingly, 
it is known that there are pairs of matrices $A$ and $B$ with identical pseudospectra 
for which the corresponding norms $\|A^{k}\|$ and $\|B^{k}\|$ for $k\geq 2$ are 
completely unrelated (see \cite[Theorem 2.3]{R} for details).  This can be attributed 
to the fact that, roughly speaking, parts of a matrix may not actively play a role when 
computing its spectral norm, e.g., see the proof of Theorem \ref{quadMinPolyThm} below.
So, in an attempt to prevent such parts from being ``hidden'' and drawing inspiration 
from the way that pseudospectra are computed, Fortier Bourque and Ransford introduced 
in \cite{BR} the notion of super-identical pseudospectra of matrices belonging to the 
same class $\M_{d}$.

Matrices $A$ and $B$ in $\M_d$ are said to have \textbf{super-identical pseudospectra} if
\begin{equation}\label{condS}
			s_{k}(zI-A)=s_{k}(zI-B)\quad\text{ for all }z\in\C, k=1,\ldots,d. 
\end{equation}
Thus, since the condition in \eqref{condR} is equivalent to 
\[	s_{d}(zI-A)=s_{d}(zI-B)\quad\text{ for all }z\in\C,	\]
the requirement in \eqref{condS} is stronger than \eqref{condR}.

It can be shown (see \cite[Theorem 3.6]{R}) that if $A$ and $B$ have
super-identical pseudospectra, then the norms of $p(A)$ and $p(B)$
are at least comparable; more specifically, 
\[	\frac{1}{\sqrt{d}}|\|p(B)\|\leq \|p(A)\|\leq\sqrt{d}\|p(B)\|	\]
holds for all polynomials $p$.  This result suggests that pairs of matrices 
having super-identical pseudospectra may be polynomially isometric, but this
need not be the case; in fact, the examples of $4 \times 4$ matrices from 
\cite{BR} and \cite{RR} mentioned in section \ref{terminologySection}
have super-identical pseudospectra but are not polynomially isometric.

On the other hand, for matrices $A,B\in\M_d$ in low dimensions $d=2$ or $d=3$,
it was shown in \cite{BR} that a \emph{sufficient} condition for $A$ and $B$
to be polynomially isometric is that $A$ and $B$ have super-identical
pseudospectra.  However, the failure of the necessity can already be seen 
by the pair of $3\times 3$ matrices in \eqref{equalSpectraWithoutMultEx}.  
For $2\times 2$ matrices, it turns out 
that the notions of identical pseudospectra, polynomial isometry, and 
super-identical pseudospectra are all equivalent.

\begin{theorem}\label{equivalenceM2}
			The following statements are equivalent for $A,B\in\M_2$.
			\begin{enumerate}
						\item\label{M2ip}
									$A$ and $B$ have identical pseudospectra.
						\item\label{M2snb}	
									$A$ and $B$ are polynomially isometric.
						\item\label{M2sip}
									$A$ and $B$ have super-identical pseudospectra.
						\item\label{M2us}
									$A$ and $B$ are unitarily similar.
			\end{enumerate}
\end{theorem}

As previously mentioned, the equivalence
``\ref{M2ip}$\iff$\ref{M2snb}'' 
was stated in \cite{GT}.  
The equivalence ``\ref{M2sip}$\iff$\ref{M2us}'' 
was established in \cite{BR}.  Our proof of Theorem \ref{equivalenceM2} is 
based on Lemma \ref{skTwoByTwoCase} below whose proof is left to the reader.  
Before stating that lemma, we need a definition.

Given a polynomial $p$, define $D_p:\C^2\to\C$ by
\begin{equation*}	
D_p(\alpha,\beta):= \left\{\begin{array}{cc} 
    p^{\prime}(\alpha) & \text{ if }\; \alpha = \beta \\
		&\\
		\frac{p(\alpha)-p(\beta)}{\alpha-\beta} & \text{ if }\; \alpha \neq \beta \end{array}\right.. 	
\end{equation*}

\begin{lemma}\label{skTwoByTwoCase}
			The family of $2\times 2$ matrices
			$t(\alpha,\beta,\delta)=\left[
			\begin{array}{ccc}
						\alpha	&	\delta\\
						0				&	\beta
			\end{array}\right]$ has the following properties for $\alpha,\beta,\delta\in\C$.
			\begin{enumerate}
						\item	$s_1(t(\alpha,\beta,\delta))=s_2(t(\alpha,\beta,\delta))$ 
									if and only if $|\alpha|=|\beta|$ and $\delta=0$.
						\item	$s_1(t(\alpha,\beta,\delta))$ and $s_2(t(\alpha,\beta,\delta))$ 
									are, respectively, strictly increasing and strictly decreasing 
									in $|\delta|$.
						\item	$s_1(t(\alpha,\beta,0))=\max\{|\alpha|,|\beta|\}$ and 
									$s_2(t(\alpha,\beta,0))=\min\{|\alpha|,|\beta|\}$.
						\item\label{interchangePart}
									$s_j(t(\alpha,\beta,\delta))=s_j(t(\beta,\alpha,\delta))$
									for $j=1,2$.
						\item	$p(t(\alpha,\beta,\delta))
									=t(p(\alpha),p(\beta),\delta D_{p}(\alpha,\beta))$
									for any polynomial $p$.
			\end{enumerate}
\end{lemma}

\begin{proof}[Proof of Theorem \ref{equivalenceM2}]
			By our preliminary remarks, it suffices to show that matrices with 
			identical pseudospectra must be unitarily similar.  Suppose $A$ and $B$ 
			have identical pseudospectra.  Then $A$ and $B$ have the same eigenvalues, 
			say $\alpha$ and $\beta$.  Consequently, 
			$A$ and $B$ are unitarily similar to upper triangular matrices of the form
			\[	\left[
					\begin{array}{rr}
								\alpha	&	\delta_{A}\\
								0				&	\beta
					\end{array}\right]\;\text{ and }\;
					\left[
					\begin{array}{rr}
								\alpha	&	\delta_{B}\\
								0				&	\beta
					\end{array}\right],	\]
			for some $\delta_A,\delta_B\geq 0$ (e.g., see \cite[Chapter 25]{H}). Since the
			singular values of a matrix are invariant under multiplication by unitary 
			matrices on the right and left, 
			\[	s_{2}(t(z-\alpha,z-\beta,-\delta_A))=\|(zI-A)^{-1}\|^{-1}	\]
			and 
			\[	s_{2}(t(z-\alpha,z-\beta,-\delta_B))=\|(zI-B)^{-1}\|^{-1}	\]
			are equal.  Hence, by Lemma \ref{skTwoByTwoCase}, $\delta_A=\delta_B$
			and so $A$ and $B$ are unitarily similar.
\end{proof}

\begin{remark}\label{M2traceCriteria}
			In light of Murnaghan's criterion \cite{M} for unitary similarity of matrices 
			in $\M_2$, an easy-to-check necessary and sufficient condition for any (or all) 
			of the statements in Theorem \ref{equivalenceM2} is that $\tr A^{*}A=\tr B^{*}B$ 
			and $\tr A^k=\tr B^k$ for $k=1,2$.  Thus, these three traces form a complete
			set of invariants to determine when matrices have identical pseudospectra.
\end{remark}

As an amusing consequence, we state the following corollary and leave its proof to 
the reader.

\begin{corollary}\label{easyCorollary}
			There are similar matrices $A,B\in\M_2$ that do not have identical
			pseudospectra.
\end{corollary}

It is worth mentioning that if two $d\times d$ matrices have super-identical pseudospectra, 
then they must be similar \cite{AGV}.  However, by Corollary \ref{easyCorollary}, the 
converse need not hold.

What about the case of $3\times 3$ matrices?  In this context, matrices having
super-identical pseudospectra need not be unitarily similar.  After all, a matrix 
$A$ and its transpose $A^{t}$ always have super-identical pseudospectra, but there 
are known examples of $3\times 3$ matrices $A$ that are \emph{not} unitarily similar 
to $A^{t}$ (see after Theorem \ref{thmUnitaryEquivalence} below).  Even better, it is 
proved in \cite{BR} that $A,B\in\M_3$ have super-identical pseudospectra if and only 
if $A$ is unitarily similar to $B$ or to $B^{t}$; consequently, $A$ and $B$ must be 
polynomially isometric.  Furthermore, in analogy to Pearcy's or Sibirskii's criteria 
for unitary similarity of $3\times 3$ matrices (see \cite{P} and \cite{S}), the 
following six trace conditions are necessary and sufficient for $3\times 3$ 
matrices $A$ and $B$ to have super-identical pseudospectra \cite{R}: 
$\tr \left(A^{*}A\right)=\tr \left(B^{*}B\right)$,
$\tr \left(A^{*}A^2\right)=\tr \left(B^{*}B^2\right)$,  
$\tr \left(A^{*2}A^2\right)=\tr \left(B^{*2}B^2\right)$, $\tr A^k=\tr B^k \;\text{ for } k=1,2, 3$.
Although these results provide characterizations for matrices having super-identical 
pseudospectra, they do not appear to answer these simple questions: If 
a pair of matrices have identical pseudospectra, what condition may ensure that they 
have super-identical pseudospectra?  Is that condition necessary and sufficient? 
We now close this gap. 
\begin{theorem}\label{thmSIPCharacterization}
			The following statements are equivalent for $A,B\in\M_3$.
			\begin{enumerate}
					\item $A$ and $B$ have identical pseudospectra and $\chi_A=\chi_B$. 
					\item $A$ and $B$ have super-identical pseudospectra. 
			\end{enumerate}
\end{theorem}

To prove Theorem \ref{thmSIPCharacterization}, we employ two lemmas.  We postpone
the proof of Lemma \ref{resolventEigenRe} to Appendix \ref{appendix} and leave
that of Lemma \ref{constantSumProductLemma} to the reader.  As usual, for $T\in\M_d$, 
$\re T=(T+T^{*})/2$.

\begin{lemma}\label{resolventEigenRe}
			Let $A$ and $B$ be square matrices (not necessarily of the same size).  
			If $A$ and $B$ have identical pseudospectra, then the largest eigenvalues 
			of the matrices $\re A$ and $\re B$ coincide, as do the smallest eigenvalues.
\end{lemma}

\begin{lemma}\label{constantSumProductLemma}
			If $x_1,y_1,x_2,y_2\in\R$ satisfy $x_1+y_1=x_2+y_2$ and 
			$x_1\cdot y_1=x_2\cdot y_2$, then either 
			$(x_1,y_1)=(x_2,y_2)$ or $(x_1,y_1)=(y_2,x_2)$.
\end{lemma}

\begin{proof}[Proof of Theorem \ref{thmSIPCharacterization}]
			If $A,B\in\M_3$ have super-identical pseudospectra (see \eqref{condS}), 
			then  $A$ and $B$ necessarily have identical pseudospectra   
			and the same characteristic polynomials, as 
			\begin{equation}\label{prodS_k}
						|\det(zI-A)|=\prod_{k=1}^{3}s_{k}(zI-A)
						=\prod_{k=1}^{3}s_{k}(zI-B)=|\det(zI-B)|.
			\end{equation}
			
			Suppose now that $A$ and $B$ have identical pseudospectra and 
			equal characteristic polynomials. Since
			\[	\sum_{k=1}^{3}s_{k}^{2}(zI-A)=\tr[(zI-A)^{*}(zI-A)]
			=3|z|^2-\bar{z}\tr A-z\overline{\tr A}+\tr A^{*}A, \]
			we see that 
			\begin{equation}\label{sumOfThreeSk}
						\sum_{k=1}^{3}s_{k}^2(zI-A)=\sum_{k=1}^{3}s_{k}^2(zI-B)
			\end{equation}
			holds provided $\tr A=\tr B$ and $\tr A^{*}A=\tr B^{*}B$.
			Clearly, $\chi_{A}=\chi_{B}$ implies $\tr A^k=\tr B^k$ for $k=1,2,3$.
			In particular, $\tr A=\tr B$ and so 
			$\tr\re A=\tr\re B$; therefore,	by Lemma \ref{resolventEigenRe}, 
			$\re A$ and $\re B$ have the same eigenvalues (counting multiplicities) 
			and so $\tr(\re A)^2=\tr(\re B)^2$, or equivalently,
			\[	\tr A^2+2\tr A^{*}A+\tr A^{*2}=\tr B^2+2\tr B^{*}B+\tr B^{*2}.	\]
			Thus, $\tr A^{*}A=\tr B^{*}B$ and \eqref{sumOfThreeSk} is established.

			Recalling that $A$ and $B$ have identical pseudospectra, \eqref{prodS_k}
			and \eqref{sumOfThreeSk} simplify to
			\[	\prod_{k=1}^{2}s_{k}(zI-A)=\prod_{k=1}^{2}s_{k}(zI-B)
					\;\text{ and }\;
					\sum_{k=1}^{2}s_{k}^2(zI-A)=\sum_{k=1}^{2}s_{k}^2(zI-B).	\]
			Hence, the fact that $A$ and $B$ have super-identical pseudospectra 
			follows now from Lemma \ref{constantSumProductLemma}.
\end{proof}

Although matrices with identical pseudospectra are known to have the same 
minimal polynomials (see Theorem \ref{ipImpliesEqualMs} below), they need 
not have the same characteristic polynomials; for an example, consider 
again the diagonal matrices $A$ and $B$ in \eqref{equalSpectraWithoutMultEx}. 
This example demonstrates that the assumption $\chi_A=\chi_B$ in Theorem 
\ref{thmSIPCharacterization} is not superfluous.  More strikingly, these
$A$ and $B$ have identical pseudospectra and yet \emph{none} of Ransford's 
six trace criteria (which characterize super-identical pseudospectra as 
found in \cite{R} and stated just before Theorem \ref{thmSIPCharacterization}) 
hold.  Hence, no five of those six traces alone suffice to characterize 
when a pair of (generic) matrices have identical pseudospectra!

Nevertheless, by Theorem \ref{thmSIPCharacterization}, Ransford's six trace 
criteria may be used to confirm whether (or not) a pair of matrices have 
identical pseudospectra and the same minimal polynomial of degree $3$. Instead, 
in the case of matrices with common minimal polynomial of degree $2$, we find 
and present in the next section another easy-to-check criterion to confirm that 
they have identical pseudospectra. 

\section{Matrices with quadratic minimal polynomials.}\label{quadMinPolySection}

In this section, we complete the proof of Theorem \ref{MainThm} via the proof of 
Theorem \ref{HalfMainThm} below.  To do so, we 
state an analog of Lemma \ref{skTwoByTwoCase} for $3\times 3$ matrices whose 
proof is left to the interested reader.

\begin{lemma}\label{skThreeByThreeCase}
			The family of matrices
			$T(\gamma,\alpha,\beta,\delta)=\left[
			\begin{array}{ccc}
						\gamma	&	0			&	0\\
						0				&	\alpha	&	\delta\\
						0				&	0			&	\beta
			\end{array}\right]$ has the following properties for $\alpha,\beta,\gamma,\delta\in\C$.
			\begin{enumerate}
						\item	Following the notation of Lemma \ref{skTwoByTwoCase},
									the singular values of $T(\gamma,\alpha,\beta,\delta)$ consist 
									of those of $t(\alpha,\beta,\delta)$ and $|\gamma|$.  In particular, 
									\begin{align*}
												s_1(T(\gamma,\alpha,\beta,\delta))&=s_1(t(\alpha,\beta,\delta))
												\;\text{ when }|\gamma|\leq\max\{|\alpha|,|\beta|\},\text{ and }\\
												s_3(T(\gamma,\alpha,\beta,\delta))&=s_2(t(\alpha,\beta,\delta))
												\;\text{ when }|\gamma|\geq\min\{|\alpha|,|\beta|\}.
									\end{align*}
						\item	$p(T(\gamma,\alpha,\beta,\delta))
									=T(p(\gamma),p(\alpha),p(\beta),\delta D_{p}(\alpha,\beta))$
									for any polynomial $p$.
			\end{enumerate}
\end{lemma}

\begin{theorem}\label{HalfMainThm}
			The following statements are equivalent for $A,B\in\M_3$.
			\begin{enumerate}
					\item $A$ and $B$ have identical pseudospectra. 
					\item $A$ and $B$ are polynomially isometric.
			\end{enumerate}
			Moreover, if $A$ and $B$ also have the same minimal polynomial of degree $2$, 
			then the above statements are equivalent to 
			\begin{enumerate} 
						\item[3.] $\|A-\gamma_{A}I\|_{F}=\|B-\gamma_{B}I\|_{F}$, where $\gamma_{A}$ 
											and $\gamma_{B}$ are the eigenvalues corresponding to $A$ and $B$, 
											respectively, of largest multiplicity.
			\end{enumerate}
\end{theorem}
\begin{proof}
			It is shown in \cite{GT} that $A$ and $B$ have identical pseudospectra 
			whenever they are polynomially isometric.  To prove the converse, 
			assume $A$ and $B$ have identical pseudospectra.  By Theorem 
			\ref{ipImpliesEqualMs} in Appendix \ref{appendix}, $A$ and $B$ have the same 
			minimal polynomials.  If their common minimal polynomial has degree one, 
			then $A=\alpha I$ and $B=\beta I$ for some $\alpha,\beta\in\C$; it follows 
			readily from this that $\alpha=\beta$ and so $A$ and $B$ are polynomially 
			isometric.  Thus, in light of Theorem \ref{thmSIPCharacterization}, we 
			assume that $A$ and $B$ have common quadratic minimal polynomial 
			$p(z) = (z-\alpha)(z-\beta)$, where $\alpha, \beta \in \C$
			and $\alpha\neq\beta$, and such that $\chi_A\neq\chi_B$; 
			i.e., $A$ and $B$ have the same eigenvalues but with distinct 
			multiplicities.  Since the three statements listed in the 
			theorem are invariant under unitary similarity, we assume 
			further without loss of generality that $A$ and $B$ are 
			upper triangular matrices of the form 
			\[	A=\left[
					\begin{array}{ccc}
								\alpha	&	0				&	0\\
								0				&	\alpha	&	\delta_{A}\\
								0				&	0				&	\beta
					\end{array}\right]\;\text{ and }\;
					B=\left[
					\begin{array}{ccc}
								\beta		&	0				& 0	\\
								0				&	\alpha	&	\delta_{B}\\
								0				&	0				&	\beta
					\end{array}\right],	\]
			where $\delta_A,\delta_B>0$ (\cite[Chapter 25]{H}).

			By Lemmas \ref{skTwoByTwoCase} and \ref{skThreeByThreeCase}, 
			$A=T(\alpha,\alpha,\beta,\delta_{A})$ and 
			$B=T(\beta,\alpha,\beta,\delta_{B})$ satisfy
			\begin{align*}
						\|(zI-A)^{-1}\|^{-1}
						&=s_{2}\left(t(z-\alpha,z-\beta,-\delta_A)\right),\\
						\|(zI-B)^{-1}\|^{-1}
						&=s_{2}\left(t(z-\alpha,z-\beta,-\delta_B)\right),\\
						\|p(A)\|
						&=s_{1}\left(t(p(\alpha),p(\beta),\delta_A D_{p}(\alpha,\beta))\right),
						\text{ and }\\
						\|p(B)\|
						&=s_{1}\left(t(p(\alpha),p(\beta),\delta_B D_{p}(\alpha,\beta))\right).
			\end{align*}
			We now see that $A$ and $B$ have identical pseudospectra if and only if 
			$\delta_A=\delta_B$.  Likewise, $A$ and $B$ are polynomially isometric
			if and only if $|\delta_A D_p(\alpha,\beta)|=|\delta_B D_p(\alpha,\beta)|$
			for all polynomials $p$, or equivalently, $\delta_A=\delta_B$.  On the other 
			hand, $\delta_A=\delta_B$ is equivalent to $\|A-\alpha I\|_{F}=\|B-\beta I\|_{F}$.  
\end{proof}

Another look at the proof of Theorem \ref{HalfMainThm} reveals the validity of the 
following result which complements Theorem \ref{thmSIPCharacterization} above.
We omit the details.

\begin{theorem}\label{thmUnitaryEquivalence}
			The following statements are equivalent for $A,B\in\M_3$ having equal quadratic 
			minimal polynomials.
			\begin{enumerate}
						\item $A$ and $B$ have identical pseudospectra and $\chi_A=\chi_B$. 
						\item	$A$ and $B$ are unitarily similar.
			\end{enumerate}
\end{theorem}

Note, however, that the equivalence in Theorem \ref{thmUnitaryEquivalence} need not
hold if $A$ and $B$ have equal \emph{cubic} minimal polynomials.  For an example, let
\[	A=\left[
		\begin{array}{ccc}
					0	&	1	& 0\\
					0	&	0	& 2\\
					0	&	0	&	0
		\end{array}\right]	\]
and $B=A^{t}$, and note that $A$ and $B$ cannot be unitarily equivalent because
\[	\tr(AA^{*}A^{2}A^{*2})\neq\tr(BB^{*}B^{2}B^{*2})	\]
(see the third example in \cite{S}).  It also worth noting that unitary similarity 
of $d\times d$ matrices having equal quadratic minimal polynomials has been 
characterized in \cite{GI} as those matrices having the same eigenvalues and the 
same singular values.

Although the third condition in Theorem \ref{HalfMainThm} is easy to check, it 
does not lend itself to generalization.  For instance, by Lemma \ref{skTwoByTwoCase}, 
the $4\times 4$ matrices
\[	A=\left[
		\begin{array}{cccc}
					1	&	4	& 0	& 0\\
					0	&	0	& 0	& 0\\
					0	&	0	& 1	& 3\\
					0	&	0	& 0	& 0
		\end{array}\right]\;\text{ and }\;
		B=\left[
		\begin{array}{cccc}
					1	&	4	& 0	& 0\\
					0	&	0	& 0	& 0\\
					0	&	0	& 1	& 2\\
					0	&	0	& 0	& 0
		\end{array}\right]	\]
have identical pseudospectra, equal quadratic minimal polynomials, and yet
the Frobenius norms $\|A-\gamma_A I\|_{F}$ and $\|B-\gamma_B I\|_{F}$ are
not equal whether one interprets $\gamma_A$ and $\gamma_B$ as $0$ or $1$.
Nevertheless, the following theorem holds.

\begin{theorem}\label{quadMinPolyThm}
			The following statements are equivalent for square matrices $A$ and $B$
			(not necessarily of the same size) with quadratic minimal polynomials.
			\begin{enumerate}
						\item	$A$ and $B$ have identical pseudospectra.
						\item	$A$ and $B$ are polynomially isometric.
						\item	$\|A\|=\|B\|$.
			\end{enumerate}
\end{theorem}
\begin{proof}[Sketch of the proof.]
			In view of the assumptions on $A$ and $B$, we may assume that $A$ and $B$ 
			have the form (\cite[Chapter 25]{H})
			\begin{align*}
						A=\alpha I_{r}\oplus\beta I_{s}\oplus
						\left[
						\begin{array}{rr}
									\alpha	&	\delta_{1}\\
									0				&	\beta
						\end{array}\right]\oplus\cdots\oplus
						\left[
						\begin{array}{rr}
									\alpha	&	\delta_{t}\\
									0				&	\beta
						\end{array}\right]\\
						\text{ and }\quad
						B=\alpha I_{u}\oplus\beta I_{v}\oplus
						\left[
						\begin{array}{rr}
									\alpha	&	\gamma_{1}\\
									0				&	\beta
						\end{array}\right]\oplus\cdots\oplus
						\left[
						\begin{array}{rr}
									\alpha	&	\gamma_{w}\\
									0				&	\beta
						\end{array}\right]
			\end{align*}
			where $\alpha$ and $\beta$ are the zeros of $m_A$ and $m_B$, 
			$I_n$ denotes the $n\times n$ identity matrix, and 
			\[	\delta_1\geq\cdots\geq\delta_t\geq0\quad\text{ and }\quad
					\gamma_1\geq\cdots\geq\gamma_w\geq0.	\]
			In this case, the conclusion follows in an analogous manner 
			as in that of Theorem \ref{HalfMainThm} after observing that
			$\|A\|=\|B\|$ precisely when $\delta_1=\gamma_1$ by Lemma 
			\ref{skTwoByTwoCase}.
\end{proof}

\appendix
\section{Two results on identical pseudospectra.}\label{appendix}

In this section, we prove two technical results used in this article on matrices 
(of arbitrary size $d$) having identical pseudospectra.  The first, stated as Lemma 
\ref{resolventEigenRe} above, concerns equality of the smallest and largest 
eigenvalues of their ``real parts.''  The second concerns equality of 
minimal polynomials.

To begin, recall that the numerical range (or field of values), $W(T)$, of 
$T\in\M_{d}$ is defined by $W(T)=\{ x^{*}Tx: \|x\|_{\C^{d}}=1 \}$.
It is well known that $W(T)$ is a compact convex subset of $\C$ that 
contains $\sigma(T)$.  As such, $W(T)$ is the intersection of all closed 
half-planes $H$ containing it.  With these notions available, we are ready 
to prove Lemma \ref{resolventEigenRe}.

\begin{proof}[Proof of Lemma \ref{resolventEigenRe}]
			By a result from \cite{BO}, a closed half-plane $H$ satisfies 
			$W(T)\subseteq H$ if and only if $\sigma(T)\subseteq H$ and 
			$\|(zI-T)^{-1}\|\leq 1/\dist(z,H)$ for all $z\notin H$.  
			Therefore, $W(A)=W(B)$ when $A$ and $B$ have identical pseudospectra, 
			and the desired conclusion is now at hand; after all, the smallest 
			and largest eigenvalues of the self-adjoint matrix $\re T$ 
			are given by, respectively,
			\[	\min_{\|x\|_{\C^{d}}=1}\langle (\re T)x,x\rangle=\min_{z\in W(T)}\re z
					\;\text{ and }\;
					\max_{\|x\|_{\C^{d}}=1}\langle (\re T)x,x\rangle=\max_{z\in W(T)}\re z.\qedhere	\]
\end{proof}

Next, we turn to the second result.  In \cite[Chapter 23]{H}, it is stated that 
matrices having identical pseudospectra must have the same minimal polynomials.  
Since we were unable to find a direct proof of this fact, we include one for 
completeness.

\begin{theorem}\label{ipImpliesEqualMs}
			Let $A$ and $B$ be square matrices (not necessarily of the same size).  
			If $A$ and $B$ have identical pseudospectra, then $A$ and $B$ have the 
			same minimal polynomial. 
\end{theorem}
\begin{proof}
			Recall that the minimal polynomial $m_{T}$ of $T \in \M_d$ is given by
			\begin{equation}\label{minPolynomial}
					m_{T}(z)=\prod_{\lambda\in\sigma(T)}(z-\lambda)^{\nu_T(\lambda)},
			\end{equation}
			where $\nu_{T}(\lambda)$ denotes the index of $\lambda \in \sigma(T)$,
			the smallest nonnegative integer such that 
			$\ker(T-\lambda I)^{\nu_{T}(\lambda)}=\ker(T-\lambda I)^{d}$. 

			Moreover, the resolvent $(zI-T)^{-1}$ of $T$ at $z$ is given by
			\begin{equation}\label{partialFractionResolvent}
						(zI-T)^{-1}=\sum_{\lambda\in\sigma(T)}\sum_{k=0}^{\nu(\lambda)-1}
						(z-\lambda)^{-(k+1)}(T-\lambda I)^{k}E_{\lambda},\quad
						z\notin\sigma(T).
			\end{equation}
			Here, $E_{\lambda_{1}}, \ldots, E_{\lambda_{s}}$ are the 
			uniquely defined orthogonal projections on $\C^{d}$ such that 
			$I=E_{\lambda_{1}}+\cdots+E_{\lambda_{s}}$ and so that
			the range $\ran(E_{\lambda_{j}})=\ker(T-\lambda I)^{\nu(\lambda_j)}$ 
			for $j=1,\ldots,s$.  (The formula in \eqref{partialFractionResolvent}
			may be verified by multiplying its right-hand side by 
			$(zI-T)=(z-\lambda)I-(T-\lambda I)$.) In light of \eqref{partialFractionResolvent}, 
			the product $|z-\lambda|^{\ell}\cdot\|(zI-T)^{-1}\|$ is bounded near 
			$\lambda\in\sigma(T)$ if and only if $\ell\geq\nu_T(\lambda)$. 

			Now suppose $A$ and $B$ are square matrices such that
			$	\|(zI-A)^{-1}\|\leq\|(zI-B)^{-1}\|$ for all $z\in\C$.
			Then $\sigma(A)\subseteq\sigma(B)$, and $\nu_{B}(\alpha)
			\geq\nu_{A}(\alpha)$ holds for all $\alpha\in\sigma(A)$ because
			$|z-\alpha|^{\nu_{B}(\alpha)}\cdot\|(zI-A)^{-1}\|$ is bounded near 
			$\alpha$.  Hence, $m_{A}$ divides $m_{B}$ by \eqref{minPolynomial}.  
			Reversing the roles of $A$ and $B$ then yields the result.
\end{proof}

\begin{remark}\label{JordanFormProof}
			One can also use the Jordan canonical form to prove Theorem 
			\ref{ipImpliesEqualMs}.  For a proof following that approach, 
			it suffices to note that $\nu_T(\lambda)$ equals the size of 
			largest Jordan block corresponding to $\lambda$ and then 
			proceed to compute the resolvent of the Jordan matrix.  
			From this, one can argue that (as in the proof above) 
			$|z-\lambda|^{\ell}\cdot\|(zI-T)^{-1}\|$ is bounded near 
			$\lambda\in\sigma(T)$ if and only if $\ell\geq\nu_T(\lambda)$.
\end{remark}

\textbf{Acknowledgments.}
This faculty-student research project was funded in part by the Seidler Scholarly 
Collaboration Fellowship during the summer of 2018.  We wish to thank the Seidler 
family and the Dean’s Office in the College of Arts and Sciences at Florida Gulf 
Coast University for providing such funding.  We would also like to thank the 
referees and editors whose comments helped improve the exposition of the article. 
Lastly, the second author wishes to thank Professor M. Embree for pointing out 
that an alternative proof to Lemma \ref{ipImpliesEqualMs} based on the Jordan 
canonical form is contained in the proof of Theorem 52.3 in \cite{TE}.

\end{document}